\documentclass[12pt,oneside,english]{amsart}
\usepackage[latin1]{inputenc}
\usepackage{amssymb}

\makeatletter

\newtheorem{proposition}{Proposition}

\newtheorem{conjecture}{Conjecture}
\newtheorem{question}{Question}
\newtheorem{corollary}{Corollary}

\usepackage{babel}

\makeatother
\begin{document}
\baselineskip=17pt

\title{Representation of positive integers by the form
$x_1...x_k+x_1+...+x_k$}

\author{Vladimir Shevelev}
\address{Department of Mathematics \\Ben-Gurion University of the
 Negev\\Beer-Sheva 84105, Israel. e-mail:shevelev@bgu.ac.il}

\subjclass{11N32.}

\begin{abstract}
 For an arbitrary given $k\geq3,$ we consider a possibility of representation of a
 positive number $n$ by the form $x_1...x_k+x_1+...+x_k,
 \enskip 1\leq x_1\leq ... \leq x_k.$ We also study a question on the smallest
 value of $k\geq3$ in such a representation.
\end{abstract}
\maketitle

\section{Introduction}
In 2002, R. Zumkeller published in OEIS the sequence A072670: "Number of ways to write
 $n$ as $ij+i+j,\enskip 0<i<=j$". This sequence possesses a remarkable property.
\begin{proposition}\label{prop1}
Positive integer $n$ is not represented by the form $ij+i+j,\enskip 0<i<=j,$ if
and only if $n=p-1,$ where $p$ is prime.
\end{proposition}
\begin{proof}
Condition $n=p-1$ is sufficient, since if $n=ij+i+j,$ then $n+1=(i+1)(j+1)$ cannot
be prime. Thus $n$ of the form $p-1$ is not represented by the form
$ij+i+j,\enskip 0<i<=j.$ \enskip Suppose that, conversely, $n$  is not represented by this
form. Show that $n+1$ is prime. If $n+1\geq4$ is composite, then $n+1=rs, \enskip
 s\geq r\geq2.$ Set $i=r-1, \enskip j=s-1.$ We have
 $$ij+i+j=(r-1)(s-1)+(r-1)+(s-1)=n+1-1=n.$$
This contradicts the supposition. So $n+1$ is prime.
 \end{proof}
In this note, for an arbitrary given $k\geq3,$ we consider a more general form
 $x_1...x_k+x_1+...+x_k,\enskip 1\leq x_1\leq ... \leq x_k.$ In particular, we study
  a question on the smallest value of $k\geq3$ in a a possible representation of $n.$
\section{Necessary condition for non-representation of n}
Denote by $\nu_k(n)$ the number of ways to write $n$ by the
$$F_k=F(x_1,...,x_k)=$$
\begin{equation}\label{1}
x_1...x_k+x_1+...+x_k,
 \enskip 1\leq x_1\leq ... \leq x_k,\enskip k\geq3.
\end{equation}
\begin{proposition}\label{prop2}
If, for a given $k\geq3,$ for $n\geq k-1$ we have $\nu_k(n)=0,$ then $n-k+3$ is
prime.
\end{proposition}
\begin{proof}
If $n-k+3\geq4$ is composite, then $n-k+3=rs, \enskip
 s\geq r\geq2.$ Set $x_i=1\enskip for \enskip i=1,...,k-2$ and $x_{k-1}=r-1, \enskip x_k=s-1.$
 We have
 $$F_k=(r-1)(s-1)+(k-2)+(r-1)+(s-1)=(n-k+3)+(k-2)-1=n.$$
This contradicts the condition $\nu_k(n)=0.$ So $n-k+3$ is prime.
\end{proof}
\begin{proposition}\label{prop3}
If $k_1<k_2$ and $\nu_{k_1}(n)>0,$ then $\nu_{k_2}(n+k_2-k_1)>0.$
\end{proposition}
\begin{proof}
By the condition, there exist $x_1,...,x_{k_1}$ such that
$$ n= x_1...x_{k_1}+x_1+...+x_{k_1},
 \enskip 1\leq x_1\leq ... \leq x_{k_1},\enskip k_1\geq3. $$
Set $y_i=1,\enskip i=1,...,k_2-k_1,$ and $y_{k_2-k_1+1}=x_1,...,y_{k_2}=x_{k_1}.$
Then we have
$$ y_1...y_{k_2}+y_1+...+y_{k_2} = x_1...x_{k_1} + k_2-k_1+x_1+...+x_{k_1}=n+
x_{k_2}-x_{k_1}.$$
\end{proof}
\begin{corollary}\label{cor1}
If $k_1<k_2$ and $\nu_{k_1}(n+k_1-3)>0,$ then $\nu_{k_2}(n+k_2-3)>0.$
\end{corollary}
\begin{corollary}\label{cor2}
If $k_1<k_2$ and $\nu_{k_2}(n+k_2-3)=0,$ then $\nu_{k_1}(n+k_1-3)=0.$
\end{corollary}
Note that, by Proposition \ref{prop2}, in Corollary \ref{cor2} the number $n$ is
prime.
\section{Cases $k=3$ and $k=4$}
Consider more detail the case $k=3,$ when
$$F_3=x_1x_2x_3+x_1+x_2+x_3, \enskip 1\leq x_1\leq x_2\leq x_3.$$
The numbers of ways to write the positive numbers by the form $F_3$ are given in
the sequence A260803 by D. A. Corneth. Note that, by Proposition \ref{prop2},
a number $n\geq2,$ could be not represented by $F_3$ only in case when $n$ is prime.
However, note that sequence of primes $p$ not represented by $F_3$ should grow fast
enough. Indeed, $p$ should not be a prime of the form
\begin{equation}\label{2}
(2t+1)m+(t+2),\enskip t,m\geq2,
\end{equation}
 where $t\equiv {0 \enskip or \enskip 2} \pmod 3.$
Indeed, in this case $p = x_1x_2x_3+x_1+x_2+x_3$ for $x_1=2,\enskip x_2=t, \enskip
 x_3=m,$ if $t\leq m,$ and for
 $x_1=2,\enskip x_2=m,\enskip x_3=t$ otherwise. Since $\gcd(2t+1, t+2)=\gcd(2(t+2)-3,
  t+2)=1,$ then, by Dirichlet's theorem, for any admissible $t\geq2,$ the progression
 (\ref{2}) contains infinitely many primes $p.$ For all these primes, $\nu_3(p) >0.$
\begin{question}\label{q1}
Is the sequence of primes $\{p\enskip|\enskip\nu_3(p)=0\}$ infinite$?$
\end{question}
However, in case of $k=4,$ in view of Corollary \ref{cor1}, to the set of
progressions (\ref{2}) one can add, for example, the following set of progressions
\begin{equation}\label{3}
(4t+1)m+(t+3),\enskip t,m\geq2.
\end{equation}
\newpage
Here $\gcd(4t+1, t+3)=\gcd(4(t+3)-11,t+3)=1,$ except for $t\equiv-3 \pmod{11}.$\enskip
Hence, for any admissible $t\geq2$ the
progression (\ref{3}) contains infinitely many primes $p.$ For such $p$ we have
$$p+k-3=p+1=2\cdot2tm+2+2+t+m=F_4$$
with $x_1=x_2=2,\enskip x_3=t,\enskip x_4=m,$ if $t\leq m,$ and $x_1=x_2=2, \enskip
 x_3=m, \enskip x_4=t,$ if $t>m.$
So for such $p,$ $\nu_4(p+1)>0.$
Therefore, and, by the observations in table in Corneth's sequence A260804 for $k=4,$
the following question has another tint.
\begin{question}\label{q2}
Is the sequence of primes $\{p\enskip|\enskip\nu_4(p+1)=0\}$ only finite$?$
\end{question}
\section{Smallest $k$ for representation of $prime+k-3$}
According to Proposition \ref{prop2}, if $m$ is not represented in the form $F_k,$
then $m-k+3$ is prime. Denote by $p_n$ the $n$-th prime. Let $m-k+3=p_n.$ Then, for
 every $n,$ it is interesting a question, for either smallest $k\geq3$ the number
 $p_n+k-3$ is represented by $F_k?$ Denote by $s(n),\enskip n\geq1,$ this smallest
 $k$ and let us write $s(n)=0,$ if $p_n+k-3$ is not represented by $F_k$ for any
 $k\geq3.$ The sequence $\{s(n)\}$ starts with the following terms (A260965):
 $$0, 0, 0, 0, 0, 0, 0, 3, 4, 3, 0, 0, 4, 0, 3, 0, 3, 3, 0, 4, 3, 3, 4, 3, $$
 \begin{equation}\label{4}
4, 0, 3, 5, 3, 4, 3, ... .
\end{equation}
\begin{conjecture}\label{con1}
The sequence $({4})$ contains only a finite number of zero terms.
\end{conjecture}
For example, a solution in affirmative of Question \ref{q2}, immediately proofs
Conjecture \ref{con1}.
Here we will concern only a question on estimates of $s(n).$
\begin{proposition}\label{prop4}
\begin{equation}\label{5}
s(n) \leq \lfloor(\log_2(p_n)\rfloor.
\end{equation}
\end{proposition}
\begin{proof}
Suppose, for a given $p_n,$ there exists $k$ such that $p_n + k - 3$ is represented
by the form $F_k.$ Then for the smallest possible $k$ such a representation we call
an \slshape optimal representation \upshape with a given $p_n.$ Let us show that in
an optimal representation all $x_i>=2.$ Indeed, let $x_1 = ... = x_u = 1$ and
$x_i >= 2$ for $u+1 <= i <= k,$ such that $p_n + k - 3 = x_{u+1}...x_k + u + x_{u+1}
 + ... + x_k$ be an optimal representation. Note that $u<k,$ otherwise $F_k = 1+k$
which is not $k-3 +$ prime. Set $k_1 = k - u; \enskip y_j = x_{u+j}.$ Then
 $p_n + k_1 - 3 = y_1...y_{k_1} + y_1 + ... + y_{k_1}.$ Since $k_1 < k,$ it
 contradicts the optimality of the form $F_k.$ The contradiction shows that all $x_i$
 in an optimal represen-
\newpage
 tation are indeed more than or equal 2. So for an optimal
 representation, $p_n + k - 3 = F_k >= 2^k + 2k$ and $2^k + k + 3 <= p_n.$ Hence
  $s(n) = k_{min} < \log_2(p_n)$ and the statement follows.
\end{proof}

Now we need a criterion for $s(n)>0.$
\begin{proposition}\label{prop5}
$s(n)>0$ if and only if either there exists $t_2\geq$ such that
$$B(t_2) = 2^{t_2} + t_2 + 3 = p_n$$
or there exist $t_2\geq0, t_3\geq1$ such that
$$B(t_2, t_3) = 2^{t_2}3^{t_3} + t_2 + 2t_3 + 3 = p_n$$
or there exist $t_2\geq0, t_3\geq0, t_4\geq1$ such that
$$B(t_2, t_3, t_4) = 2^{t_2}3^{t_3}4^{t_4} + t_2 + 2t_3 + 3t_4 + 3 = p_n,$$ etc.
\end{proposition}
\begin{proof}
Distinguish the following cases for $x_i \geq 2, i=1,...,k,$ and $F_k = x_1...x_k +
x_1 + ... + x_k:$\newline

(i) All $x_i = 2, i=1,...,t_2.$ Here $k=t_2$ and $F_k=2^{t_2} + 2t_2.$ If this is
$t_2 - 3 + p_n,$ then $p_n = 2^{t_2}+t_2+3 = B(t_2).$\newline
(ii) The first $t_2$ consecutive $x_i = 2$ and $t_3$ consecutive $x_i = 3.$ Note
that $t_3 \geq 1$ (otherwise, we have case (i)). Here $k=t_2+t_3$ and
$F_k=2^{t_2}3^{t_3} + 2t_2 + 3t_3.$ If this is $k - 3 + p_n = t_2 + t_3 -3 + p_n,$
then $p_n = 2^{t_2}3^{t_3} +
t_2 + 2t_3 + 3 = B(t_2, t_3),$ etc.
\end{proof}
Note that in the expressions $B(t_2), B(t_2, t_3),... $ defined in Proposition \ref{prop5},
we can consider only the case when the last variable is positive. Indeed, in $B(t_2), \enskip t_2\geq1$ and if $t_{j+1}=0,$ then, evidently, $B(t_2,...,t_j, 0)= B(t_2,...,t_j).$
\begin{corollary}\label{cor3}
 If $v<j$ is the smallest number such that, for some $t_2,...,t_v,t_j$
$B(t_2,...,t_v,t_j) = p_n,$ then $s(n) = t_2 + ... + t_v + t_j.$ If, for a given
 $n,$ for any $j$ there is no such $v,$ then $s(n)=0.$
\end{corollary}
Practically, using this algorithm for different $j$ (cf. Section 5), we rather quickly
reduce the number of variables $t_i$ for the evaluation of $s(n).$
\section{Cases of $p_n=97$ and $p_n=101$}
Here we show that, for $p_{25}=97, p_{26}=101,$ we have $s(25)=4$ and $s(26)=0.$
Note that $B(0,0,...,0,t_j)=2(j+1)$ and, for $j\geq3,$ $B(t_2,0,...,0,t_j)=(2^{t_2}+1)j+t_2+2.$ For $t_2=1,,...,5,$ we have $3j+3,5j+4,9j+5,17j+6,33j+7$ respectively.
None of these expressions is equal to 97 or 101.
\newpage
 Further, for $j\geq4,$
$B(t_2,t_3,0,...,0,t_j)=(2^{t_2}3^{t_3}+1)j+t_2+2t_3+2.$ Here $t_2>0,$ otherwise we have even values. For $(t_2,t_3)=(1,1),(2,1),
(3,1),$ we have $7j+5,13j+6,25j+7$ respectively. None of these expressions is equal
to 97 or 101, expect for $13j+6=97$ for $j=7$ which corresponds to
$t_2=2, t_3=1,t_7=1.$ Hence, by Corollary \ref{cor3}, $s(25)=2+1+1=4.$
Continuing the research for $p=101,$ note that, for $j\geq5,$
$B(t_2,t_3,t_4,0,...,0,t_j)=(2^{t_2}3^{t_3}4^{t_4}+1)j+t_2+2t_3+3t_4+2.$ Here already
for: $(t_2,t_3,t_4)=(1,1,1)$  we have $25j+8>101.$ It completes the case
$t_j=1.$ In case $t_j=2$ we have $B(t_2,0,...,0,t_j)=2^{t_2}j^2+t_2+2(j-1)+3,\enskip
j\geq3.$ Here $t_2$ should be even (otherwise $B(t_2,0,...,0,t_j)$ is even).
For $t_2=2,4,$ we have $4j^2+2j+3,16j^2+2j+5.$ respectively. None of these expressions is equal 101. For $j\geq4,$ $B(t_2,t_3,0,...,0,t_j)=2^{t_2}3^{t_3}j^2+t_2+2t_3+2(j-1)+3$ is
$\geq108$ already for $t_2=t_3=1.$ Finally, in case $t_j\geq3,\enskip j\geq3$ we have $B(t_2,0,...,0,t_j)=64$ for $t_2=1,j=3, t_j=3$ and $>101$ otherwise. So, $s(26)=0.$

\end{document}